\documentclass[a4paper,12pt]{amsart}

\usepackage{amssymb}
\usepackage[normalem]{ulem}
\usepackage[all]{xy}
\usepackage{color}

\definecolor{darkblue}{rgb}{0.0,0.0,0.7}

\newcommand\field[1]{\mathbb{#1}}
\newcommand\CC{\field{C}}
\newcommand\NN{\field{N}}

\newcommand\RR{\field{R}}

\newcommand\Bb{\mathcal B}
\newcommand\Dd{\mathcal D}
\newcommand\Ff{\mathcal F}

\newcommand\go{G^{(0)}}

\renewcommand\ker{\operatorname{ker}}

\newcommand\lsp{\operatorname{span}}

\newcommand\supp{\operatorname{supp}}

\newcommand{\inv}{^{-1}}

\newcommand\I[1]{\operatorname{Ideal}[#1]}
\newcommand{\minimatrix}[4]{\bigl(\begin{smallmatrix}#1 & #2 \\ #3 & #4\end{smallmatrix} \bigr)}

\newcommand{\cs}{\ensuremath{C^{*}}}

\theoremstyle{plain}
\newtheorem{theorem}{Theorem}[section]
\newtheorem*{theorem*}{Theorem}
\newtheorem*{prop*}{Proposition}
\newtheorem{cor}[theorem]{Corollary}
\newtheorem{lem}[theorem]{Lemma}
\newtheorem{prop}[theorem]{Proposition}

\theoremstyle{remark}
\newtheorem{rmk}[theorem]{Remark}

\theoremstyle{definition}

\numberwithin{equation}{section}

\title{Purely infinite  $\cs$-algebras associated to \'etale groupoids}

\author[J. Brown]{Jonathan Brown}
\address{Jonathan Brown\\Mathematics Department\\
    Kansas State University\\
138 Cardwell Hall\\
Manhattan, KS 66506-2602\\
USA.}
\email{brownjh@math.kansas.edu}

\author[L. O. Clark]{Lisa Orloff Clark}
\address{Lisa Orloff Clark\\Department of Mathematics and Statistics\\
University of Otago\\
 PO Box 56\\ Dunedin
Dunedin 9054\\
New Zealand.}
\email{lclark@maths.otago.ac.nz}

\author{Adam Sierakowski}
\address{Adam Sierakowski\\
    School of Mathematics and Applied Statistics\\
    University of Wollongong\\
    NSW 2522\\
    Australia}
\email{asierakowski@uow.edu.au}

\date{\today}

\thanks{This research was supported by the Australian Research Council.}


\begin{document}

\maketitle

\begin{abstract}
Let $G$ be
a Hausdorff, \'etale groupoid that is minimal and topologically principal.
We show  that $\cs_r(G)$ is purely infinite simple  if and only if all the nonzero
positive elements of $C_0(\go)$ are infinite in $C_r^*(G)$.
If $G$ is a Hausdorff, ample groupoid, then we show that $C^*_r(G)$ is
purely infinite simple if and only if every nonzero projection in $C_0(\go)$ is infinite in $C^*_r(G)$.
We then show how this result applies to $k$-graph $\cs$-algebras.  Finally, we investigate strongly
purely infinite groupoid $C^*$-algebras.
\end{abstract}

\section{Introduction}

Purely infinite simple $C^*$-algebras were introduced by Cuntz in \cite{cuntz81} where he showed that the $K_0$
group of such algebras can be computed within the algebra itself without resorting to the usual direct
limit construction. The $K$-theory groups of $C^*$-algebras have long been known to be computable
invariants and Cuntz's result shows that this computation is easier when the $C^*$-algebra is purely
infinite simple.  Elliott initiated a program to find a suitably large class of
$C^*$-algebras on which the $K$-theory groups provide a complete isomorphism invariant (see \cite{ell93}).  This program
has achieved remarkable success, most notably in a theorem of Kirchberg and Phillips \cite{Kirchberg,Phil00} which
states that every  \emph{Kirchberg algebra} satisfying
the Universal Coefficient Theorem (UCT) is classified by the isomorphism classes of its ordered $K$-theory groups.
A Kirchberg algebra is a separable, nuclear, purely infinite simple $C^*$-algebra.

The allure of classification via the Kirchberg-Phillips theorem has lead many authors to
study when various constructions of $C^*$-algebras yield purely infinite simple algebras.  Kumjian, Pask and
Raeburn show that a graph $C^*$-algebra of a cofinal graph is purely infinite simple if and only if
every vertex can be reached from a loop with an entrance \cite[Theorem~3.9]{KPR98}. Carlsen and Thomsen show that if the
$C^*$-algebra constructed from a locally injective surjection $\theta$ on a compact metric space of
finite covering dimension is simple, then it is purely infinite simple if and only if $\theta$ is not a
homeomorphism \cite[Corollary~6.6]{CT}. R{\o}rdam and Sierakowski \cite{RS12} show that if a
countable exact group $H$ acts by an essentially free action on the Cantor set $X$
and the type semigroup of clopen subsets of $X$ is almost unperforated, then $C_0(X)\rtimes_r H$ is purely
infinite if and only if every clopen set $E$ in $X$ is paradoxical. The constructions in each of the above
examples are special cases of groupoid $C^*$-algebras.

In this paper we investigate purely infinite $C^*$-algebras associated
to Hausdorff, \'etale groupoids.   In sections~3 through 5 we restrict our attention to simple groupoid $C^*$-algebras.
Characterising simplicity of groupoid $C^*$-algebras is known
and we readily make use of the following theorem from \cite[Theorem~5.1]{BCFS}:
\begin{theorem}[Brown-Clark-Farthing-Sims]
Let $G$ be a second-countable, locally compact, Hausdorff and \'etale groupoid. Then
$C^*(G)$ is simple if and only if all of the following conditions are satisfied.
\begin{enumerate}
\item\label{it:metrically amenable} $C^*(G) = C^*_r(G)$;
\item\label{it:discretely trivial} $G$ is topologically principal;
\item\label{it:minimal} $G$ is minimal.
\end{enumerate}

\end{theorem}
However, necessary and sufficient conditions on the groupoid for the associated algebra to be
\emph{purely infinite} simple are not known.  Anantharaman-Delaroche showed that
`locally contracting' is a sufficient condition on the groupoid in \cite[Proposition~2.4]{A-D97} but
whether locally contracting is {necessary} remains an open question.
Part of the difficulty in characterizing those groupoids that give rise to purely infinite simple $C^*$-algebras
is relating arbitrary projections in the groupoid $C^*$-algebra to the groupoid itself.

In Section~\ref{sec:topprin} we look at a necessary and sufficient conditions for ensuring
 pure infiniteness of groupoid $\cs$-algebras.  We show that for a Hausdorff, \'etale, topologically principal, and minimal groupoid $G$ the \cs-algebra
$\cs_r(G)$ is purely infinite simple  if and only if it all the nonzero
positive elements of $C_0(\go)$ are infinite in $C_r^*(G)$ and (see Theorem~\ref{thm:general}).  In Section~4 we specialize to Hausdorff,
\emph{ample} groupoids.
This is an important class of examples because every Kirchberg algebra in UCT is Morita equivalent to an
algebra associated to an Hausdorff, ample groupoid (see \cite{spielberg2007}).
We show in Theorem~\ref{thm:totdis} for a Hausdorff, ample groupoid $G$, that is also topologically principal and
minimal, the \cs-algebra $C^*_r(G)$ is purely infinite if and only if every nonzero projection in $C_0(\go)$ is infinite in $C^*_r(G)$.
Theorem~\ref{thm:totdis} is a generalisation of \cite{GS} about partial actions.
In Section~\ref{Sec:kgraph} we demonstrate how Theorem~\ref{thm:totdis} applies to $k$-graph $\cs$-algebras.

In Section~\ref{sec:NS} we turn our attention to the non-simple case.  In \cite{KR-infty},
Kirchberg and R{\o}rdam introduce three separate notions of purely infinite $C^*$-algebras: weakly purely infinite,
purely infinite and strongly purely infinite.  Of these notions, the last one appears to be the
most useful in the classification theory of non-simple $\cs$-algebras.  Indeed, Kirchberg showed in \cite{Kir04} that two separable,
nuclear, strongly purely infinite $C^*$-algebras with the same primitive ideal space $X$ are isomorphic if and only if they are
$KK_X$-equivalent.  We provide a characterization of when groupoid $\cs$-algebras are strongly purely infinite in Proposition~\ref{cor.1.2}.


\section{Preliminaries}

\subsection{Groupoids}A  {\em groupoid} $G$ is a small category in which every morphism is invertible.
The set of objects in $G$ is identified with the set of identity morphisms and both are denoted by $\go$.
We call $\go$ the {\em unit space} of $G$. Each morphism $\gamma$ in the category has a range and source denoted
$r(\gamma)$ and $s(\gamma)$ respectively and thus $r$ and $s$ define maps $G\to \go$.

A {\em topological groupoid} is a groupoid with a topology in which composition is continuous and inversion
is a homeomorphism. An {\em open bisection} in a topological groupoid $G$ is an open set $B$ such that both
 $r$ and $s$ restricted to $B$ are homeomorphisms; in particular these restrictions are injective.
An {\em \'etale groupoid} is a topological groupoid where $s$ is a local homeomorphism.
If a groupoid $G$ is Hausdorff and \'etale, then the unit space $\go$ is open and closed in $G$.
If $G$ is a locally compact, Hausdorff groupoid, then $G$ is \'etale if and only if
there is a basis for the topology on
$G$ consisting of open bisections with compact closure.
A topological groupoid is called \emph{ample} if it has a basis of compact open bisections.   If $G$ is
locally compact, Hausdorff and \'etale groupoid, then we note that $G$ is ample if and only if $\go$ is
totally disconnected (see
\cite[Proposition~4.1]{Exel10}).

For a subsets $L,K\subseteq G$, denote $LK=\{\gamma: \gamma=\xi\zeta\quad{\text{with~}} \xi\in L,
\zeta\in K, s(\xi)=r(\zeta)\}$. With a slight abuse of notation, for $u\in \go$,
we write $uG$ and $Gu$ for $\{u\}G$ and $G\{u\}$ respectively and denote by $uGu$  the
set \[\{\gamma\in G: r(\gamma)=s(\gamma)=u\}.\]
A  topological groupoid $G$ is {\em topologically principal} if the set $\{u\in \go: uGu=\{u\}\}$
is dense in $\go$, and {\em minimal} if $G\cdot u:=\{r(\gamma): s(\gamma)=u\}$ is dense in $\go$ for
all $u\in \go$. Recall, for a second countable, locally compact, Hausdorff, \'etale groupoid $G$ the algebra $C^*(G)$ is
simple if and only if $G$ is minimal, topologically principal, and $C^*(G)=C^*_r(G)$.

\subsection{Groupoid $C^*$-algebras}Let $G$ be locally compact, Hausdorff
 \'etale groupoid and let $C_c(G)$
denote the set of compactly supported continuous functions from $G$ to $\CC$.   Since every element $\gamma$ of $G$
has a neighbourhood $B_\gamma$ such that $r|_{B_\gamma}$ is injective, the set $r\inv(u)$ is discrete for
every $u\in \go$.  Thus if $f\in C_c(G)$ then $\supp(f)\cap r\inv(u)$ is finite for all $u\in\go$.  With this,
we are able to define a convolution and involution on $C_c(G)$ such that  for $f, g\in C_c(G)$,
\[
f*g(\gamma):=\sum_{r(\eta)=r(\gamma)} f(\eta)g(\eta\inv \gamma)\quad\text{and}\quad f^*(\gamma):=\overline{f(\gamma\inv)}.
\]
Under these operations, $C_c(G)$ is a $*$-algebra.  Next define for $f\in C_c(G)$,
\begin{align*}
\|f\|_I: &=\sup_{u\in \go}\{\max \{\sum_{\gamma\in Gu} |f(\gamma)|, \sum_{\gamma\in uG} |f(\gamma)|\}\} \quad\text{and}\\
\|f\|: &= \sup\{ \|\pi(f)\|: \pi\text{~ is a $\|\cdot\|_I$-decreasing representation}\}.
\end{align*}
Then $C^*(G)$ is the completion of $C_c(G)$ in the $\|\cdot\|$-norm.

Given a unit $u\in \go$, the regular representation
$\pi_u$ of $C_c(G)$ on $\ell^2(Gu)$ associated to $u$ is characterized by
\[
\pi_u(f)\delta_\gamma=\sum_{s(\eta)=r(\gamma)}f(\eta) \delta_{\eta\gamma}.
\]
The reduced $C^*$-norm on $C_c(G)$ is $\|f\|_r=\sup\{ \|\pi_u(f)\|: u\in \go\}$ and $C^*_r(G)$ is the completion
of $C_c(G)$ in the $\|\cdot\|_r$-norm.  Our attention will be focused on the
reduced $\cs$-algebra and situations where the reduced and full algebras coincide.
Also, we will often be assuming $G$ is second-countable, implying $\cs_r(G)$ is separable \cite[Page~59]{Renault}.

Below we recall a few standard results (we heavily use) and their proofs to familiarise the reader with locally compact, Hausdorff \'etale groupoids.

\begin{lem}[{cf.\ \cite{Renault}}]
\label{lem:new}
Let $G$ be a locally compact, Hausdorff and \'etale groupoid. Then
\begin{enumerate}
\item \label{it1:review} The extension map from $C_c(\go)$ into $C_c(G)$ (where a function is defined to be zero on $G-\go$) extends to an
embedding of  $C_0(\go)$ into  $C^*_r(G)$.
\item\label{it2:review} The restriction map $E_0: C_c(G)\to C_c(\go)$ extends to a  conditional
expectation\footnote{Recall: A {\em conditional expectation} $E\colon A \to B$ is a contractive,
linear, completely positive map such that for every $b\in B, a\in A$ we have
  $E(b)=b$, $E(ba)=bE(a)$ and
$E(ab)=E(a)b$ , see \cite[II.6.10.1]{Bla}.} $E: C^*_r(G)\to C_0(\go)$.
\item\label{it3:review} The map $E$ from item (\ref{it2:review}) is faithful.  That is,
$E(b^*b)=0$ implies $b=0$ for $b \in \cs_r(G)$.
\item\label{it4:review} For every closed invariant set $D \subseteq \go$ we have the following commuting diagram:
\begin{equation*}
\label{com diag000}
\xymatrix{0 \ar[r] & C^*_r(G|_{U}) \ar[d]_{E_{U}} \ar[r]^-{\iota_r} & C^*_r(G) \ar[d]_E \ar[r]^-{\rho_r}& C^*_r(G|_{D}) \ar[d]_{E_{D}} \ar[r]& 0\\
0 \ar[r] & C_0({U}) \ar[r]^-{\iota_0} & C_0(\go) \ar[r]^-{\rho_0}& C_0({D}) \ar[r]& 0},
\end{equation*}
where $U = \go - D$, $\iota_r$ and $\rho_r$ are determined on continuous functions by extension and restriction respectively.
Moreover, $\operatorname{image}(\iota_r)\subseteq \ker\rho_r$.
\item\label{it5:review} The subalgebra $C_c(\go)$ contains an approximate unit for $C^*_r(G)$.
\end{enumerate}
\end{lem}
\begin{proof}
(\ref{it1:review}) Since $G$ is Hausdorff and \'etale, $\go$ is open and closed in $G$. Thus, the map
$C_c(\go)$ into $C_c(G)$ is well defined. For $f, g\in C_c(\go)$,
a quick computation gives
\[
f*g(\gamma)=\begin{cases} f(\gamma)g(\gamma), & \text{if}\quad \gamma\in \go;\\ 0, & \text{otherwise},\end{cases}\]
so the map from $C_c(\go)$ into $C_c(G)$ is a $^*$-homomorphism. We claim the map is isometric, that is, we claim the
reduced norm agrees with the infinity norm for
functions in $C_c(\go)$. By evaluating at point masses in $\ell^2(Gu)$,
one can show that $\|f\|_{\infty}\leq \|f\|_r$ for $f\in C_c(G)$.  The reverse inequality can be verified for $f\in C_c(\go)$ and the claim follows.
Thus the $^*$-homomorphism from $C_c(\go)$ into $C_c(G)\subseteq C^*_r(G)$ extends by continuity to an
 isometric (hence injective) $^*$-homomorphism from $C_0(\go)$ into $C^*_r(G)$.

(\ref{it2:review}) Once again using that $G$ is Hausdorff and \`etale, we have that $\go$ is open and closed in $G$ and hence
$E_0$ is well defined. One may easily verify that $E_0$ is (a) positive (b) linear
 (c) idempotent, and (d) of norm one.
Therefore $E_0$ extends by continuity to a map $E: C^*_r(G)\to C_0(\go)$ with the same properties (a)--(d).
By \cite[II.6.10.1]{Bla} we conclude that $E$ is a conditional expectation.

(\ref{it3:review}) Let $b\in C^*_r(G)$ such that $E(b^*b)=0$. We need to show that $b=0$. Let $V_\gamma: \CC\to \ell^2(G_{s(\gamma)})$ be given by $c\mapsto c\delta_\gamma$.  Then $V_\gamma^*\omega=\omega(\gamma)$.  
Since $\|b\|_r=\sup_{u\in \go}\|\pi_u(b)\|$ and 
\begin{align*}
\|\pi_u(b)\delta_\gamma\|^2&=\langle \pi_u(b)\delta_\gamma, \pi_u(b)\delta_\gamma\rangle=\langle \pi_u(b^*b)\delta_\gamma,\delta_\gamma\rangle=V_\gamma^*\pi_u(b^*b)V_\gamma 1,
\end{align*}
it suffices to show that $V_\gamma^*\pi_u(b^*b)V_\gamma=0$ for all $u\in \go$ and $\gamma\in G$.

For $f\in C_c(G)$, $u\in \go$, and $c\in \CC$, we have
\begin{equation}
\label{Vu}
V_u^*\pi_u(f)V_uc=V_u^*\pi_u(f)c\delta_u=V_u^* (\sum_{s(\eta)=u} f(\eta)c\delta_\eta)=f(u)c=E(f)(u)c.
\end{equation}
Thus by the continuity of $E$, for all $a\in C^*_r(G)$,  $E(a)(u)=V_u^*\pi_u(a)V_u$ as operators on $\CC$.  

For every open bisection $B$ and $\gamma\in B$,  pick a function $\phi_{\gamma,B}\in C_c(G)$ such that $\phi_{\gamma,B}(\gamma)=1$, 
$\supp(\phi_{\gamma,B})\subseteq B$, and $0\leq \phi_{\gamma, B}\leq 1$. Now if $f\in C_c(G)$ and 
$B$ is an open bisection with $\gamma\in B$, then 
\begin{align}
\label{invar}
(E(\phi_{\gamma,B}^*f\phi_{\gamma,B}))(u)&=\sum_{r(\xi)=r(\zeta)=u} \phi_{\gamma,B}(\xi\inv)f(\xi\inv\zeta)\phi_{\gamma,B}( \zeta\inv)\nonumber,\\
\intertext{which is zero unless $\xi, \zeta\in B\inv$. Since $r(\xi)=r(\zeta)=u$, we have that $\xi=\zeta$ is the unique element of $uB\inv$. So}
(E(\phi_{\gamma,B}^*f\phi_{\gamma,B}))(u)&=\phi_{\gamma,B}(\zeta\inv)f(s(\zeta))\phi_{\gamma,B}(\zeta\inv)\leq E(f)(s(\zeta))\leq \|E(f)\|_{\infty}.
\end{align}
Now if $a\in C^*_r(G)$ then $\phi_{\gamma,B}^*a^*a\phi_{\gamma,B}$ is positive so $E(\phi_{\gamma,B}^*a^*a\phi_{\gamma,B})\geq 0$.  Therefore by 
the continuity of $E$ we can apply \eqref{invar} to obtain 
\[
0\leq E(\phi_{\gamma,B}^*b^*b\phi_{\gamma,B})\leq \|E(b^*b)\|_{\infty}=0.
\]
Thus $E(\phi_{\gamma,B}^*b^*b\phi_{\gamma,B})=0$ for all open bisections $B$ and $\gamma\in B$.

For $\gamma\in G$ pick an open bisection $B$ such that $\gamma\in B$.  Notice for $c\in \CC$
\[
\pi_{s(\gamma)}(\phi_{\gamma,B})V_{s(\gamma)}c=\pi_{s(\gamma)}(\phi_{\gamma,B})c\delta_{s(\gamma)}=\sum_{s(\eta)=s(\gamma)}\phi_{\gamma,B}(\eta)c\delta_\eta=c\delta_\gamma=V_\gamma c.
\]
Thus $\pi_{s(\gamma)}(\phi_{\gamma,B})V_{s(\gamma)}=V_\gamma$ as operators.  Now by equation~\eqref{Vu}  and the above observation we get for all $\gamma\in G$ that
\[
V_\gamma^*\pi_u(b^*b)V_\gamma= V_{s(\gamma)}^*\pi_{s(\gamma)}(\phi_{\gamma,B}^*b^*b\phi_{\gamma,B})V_{s(\gamma)}^*=E(\phi_{\gamma,B}^*b^*b\phi_{\gamma,B})=0
\]
as desired.  Therefore $b=0$ and hence $E$ is faithful. 

(\ref{it4:review}) The diagram commutes when restricting to continuous functions with compact support.
 Commutativity then passes to the respective completions by continuity. Since we know
$\rho_r(\iota_r(f))=0$ for all $f\in C_c(G|_U)$ we obtain $\operatorname{image}(\iota_r)\subseteq \ker\rho_r$ by continuity.

(\ref{it5:review}) Let $\mathcal{C}$ be the set of compact sets in $\go$ ordered by inclusion.  
For each $C\in \mathcal{C}$  pick a function $e_C$ in $C_c(\go)$ such that $0\leq e_c\leq 1$ and $e_C|_C\equiv 1$.  
Fix $f\in C_c(G)$ vanishing outside a compact set $K\subseteq G$.  For $C$ such that $s(K)\subset C$, $f*e_c=f$.  
It follows that $(e_C)$ is an approximate unit for $C^*_r(G)$.
\end{proof}


\subsection{Purely infinite simple $C^*$-algebras}

Given a $C^*$-algebra $A$ we denote its positive elements by $A^+$.  If $B$ is a subalgebra
of $A$ then $B^+\subseteq A^+$.  In particular, if $C_0(X)$ is an abelian subalgebra of $A$
and $f\in C_0(X)$ such that $f(x)\geq 0$ for all $x\in X$, then $f\in A^+$.

For positive elements $a\in M_n(A)$
and $b\in M_m(A)$, $a$ is {\em Cuntz below} $b$, denoted $a\precsim b$, if there exists a
sequence of elements $x_k$ in $M_{m,n}(A)$ such that $x_k^* b x_k\to a$ in norm. Notice
that $\precsim$ is transitive: if $a\precsim b$ and $b\precsim c$ there exist sequences
of element $x_n$ and  $y_n$ such that $x_n^* b x_n\to a$ and $y_n^* c y_n\to b$ in norm,
so $x_n^*y_n^* c y_n x_n\to a$ in norm, that is $a\precsim c$.    We say $A$ is {\em purely infinite}
if there are no characters on $A$ and for all $a,b\in A^+$, $a\precsim b$ if and only if
$a\in \overline{AbA}$ \cite[Definition~4.1]{KR}. A nonzero positive element $a\in A$ is
{\em properly infinite}  if $a\oplus a\precsim a$.  By \cite[Theorem~4.16]{KR}  $A$ is
purely infinite if and only if every nonzero positive element in $A$ is properly infinite.

A projection $p$ in a $C^*$-algebra $A$ is {\em infinite} if it is Murray-von Neumann
equivalent to a proper subprojection of itself, i.e., if there exists a partial isometry
$s$ such that $s^*s=p$ but $ss^*\lneq p$. By \cite[Proposition~4.7]{KR} a $C^*$-algebra $A$
is purely infinite if every nonzero hereditary $\cs$-subalgebra in every quotient of $A$
contains an infinite projection. For simple $\cs$-algebras the converse is also true, thus
a simple $C^*$-algebra is purely infinite  precisely when every hereditary subalgebra
contains an infinite projection.

\section{Topologically principal groupoids and positive elements of $C_0(\go)$}\label{sec:topprin}

In this section we consider, locally compact, Hausdorff and \'etale
groupoids.
We will show that we can determine when $\cs_r(G)$ is purely infinite simple by
restricting our attention to
elements of $C_0(\go)$ (see  Theorems~\ref{thm:totdis} and \ref{thm:general}).
Before we do that, we need the following technical lemmas.

\begin{lem}
\label{lem: nice func}
 Let $G$ be a locally compact, Hausdorff and \'etale groupoid and
$E:\cs_r(G) \to C_0(\go)$ be the faithful conditional expectation extending restriction.  Suppose that
$G$ is topologically principal.
 \label{lem:star}  For every $\epsilon >0$ and $c \in \cs_r(G)^+$, there exists $f \in C_0(\go)^+$ such that:
\begin{enumerate}
 \item \label{it:*1} $\|f\| = 1$;
 \item \label{it:*2} $\|fcf - fE(c)f\| < \epsilon$;
 \item \label{it:*3} $\|fE(c)f\| > \|E(c)\| - \epsilon.$
\end{enumerate}
\end{lem}

\begin{proof}
Let $\epsilon >0$. For $c=0$ the result is trivial so let $c \in {\cs_r(G)^+}$ such that $c \neq 0$.  Define
\[
 a:= \dfrac{c}{\|E(c)\|}.
\]
To find an appropriate $f$, we use the construction in the proof of \cite[Proposition~2.4]{A-D97};
we include the details below.
Find $b \in C_c(G)\cap\cs_r(G)^+$ so that $\|a-b\| < \dfrac{\epsilon}{2\|E(c)\|}$.
Then
\[\|E(b)\| > 1 - \dfrac{\epsilon}{2\|E(c)\|}\]
because $E$ is linear and $\|E(a)\|=1$.  Now, let $K:= \supp(b - E(b))$, which is
a compact subset of $G \setminus \go$.   Let
\[
U:= \{u \in \go \mid {E(b)}(u) > 1-\dfrac{\epsilon}{2\|E(c)\|}\}.
\]
Since $G$ is topologically principal,
 \cite[Lemma~2.3]{A-D97} implies
that there exists a nonempty open set $V \subseteq U$ such that $VKV = \emptyset$.  Using regularity,
fix a nonempty open set $W$ such that $\overline{W} \subseteq V$. Using normality, select a positive (nonzero)
real-valued function
$f \in C_c(\go)$
such that $f|_{\overline{W}}=1,~ \supp(f)\subseteq V$, and $0\leq
f(x)\leq 1$ for all $x \in \go$.
 Therefore, $f$ is positive in $\cs_r(G)$  and
satisfies item~(\ref{it:*1}).

To see that item~(\ref{it:*2}) holds,
a direct computation gives
\begin{equation}
\label{eq:b eq}
fbf = fE(b)f.
\end{equation}
Since  $\|a-b\| < \dfrac{\epsilon}{2\|E(c)\|}$, $\|f\|=1$ and $E$ is norm decreasing we have
\begin{equation}
\label{eq:E small}
\|fE(a)f-fE(b)f\|<\dfrac{\epsilon}{2\|E(c)\|}.
\end{equation}
Combining equations~\eqref{eq:b eq} and \eqref{eq:E small}  we get
\[
 \|faf - fE(a)f\| = \|faf -fbf+fbf-fE(b)f+fE(b)f - fE(a)f\|  < \dfrac{\epsilon}{\|E(c)\|}.\]
Thus multiplying by $\|E(c)\|$ gives
$\|fcf - fE(c)f\| < \epsilon$ as needed {in (\ref{it:*2})}.

To see item~(\ref{it:*3}) notice that since $\supp f \subseteq U$ we have
\[
fE(b)f\geq(1-\dfrac{\epsilon}{2\|E(c)\|})f^2.
\]
Since $\|f\|=1$, from the above equation and equation \eqref{eq:E small} we get
\[
\|fE(a)f\|> \|fE(b)f\|-\dfrac{\epsilon}{2\|E(c)\|}\geq 1-\dfrac{\epsilon}{2\|E(c)\|}-\dfrac{\epsilon}{2\|E(c)\|}=1-\dfrac{\epsilon}{\|E(c)\|}.
\]
Multiplying by $\|E(c)\|$ we obtain $\|fE(c)f\| >   \|E(c)\| - \epsilon$ as needed.\end{proof}

\begin{lem}
 \label{lem:!}
Let $G$ be a locally compact, Hausdorff and \'etale groupoid and
$E:\cs_r(G) \to C_0(\go)$ be the faithful conditional expectation extending restriction.  Suppose that
$G$ is topologically principal.  For every nonzero $a \in \cs_r(G)^+$, there exists {nonzero} $h \in C_0(\go)^+$ such that
$h\precsim a$.
\end{lem}

\begin{proof}
 Let $a \in \cs_r(G)^+$ such that $a \neq 0$.  Since $E$ is faithful, $E(a)$ is nonzero.
Applying Lemma~\ref{lem:star} to $c:= \frac{a}{\|E(a)\|}$ and
$\epsilon = 1/4$ gives us an $f \in C_0(\go)$ such that items (\ref{it:*1}), (\ref{it:*2}) and (\ref{it:*3}) of Lemma~\ref{lem:star}
hold. In particular $\|fE(c)f\|>\frac{3}{4}$.

Following {\cite[p. 640]{KR}}, for each $d \in {C_0(\go)^+}$ we define the element
\[(d-1/2)_+ := \phi_{1/2}(d) \in C_0(\go)^+ \]
where $\phi_{1/2} (t) = \operatorname{max} \{t-1/2, 0\}$ for $t \in \RR^+$.
Notice that
\[
\|\phi_{1/2}(d)\| = \max\{\|d\|-1/2,0\},
\]
for each $d\in C_0(G^{(0)})^+$.

Now let  $h:= (fE(c)f - 1/2)_+ \in C_0(\go)^+$.  Using item~(\ref{it:*2}) of Lemma~\ref{lem:star} and
{\cite[Lemma~2.2]{KR-infty}}, we can find $g \in \cs_r(G)$  so that
$h = g^*fcfg$.  Therefore $h \precsim a$. Finally,  $h \neq 0$ since
\[ \|h\| = \|(fE(c)f - 1/2)_+\| \geq \|fE(c)f\| - 1/2 \geq 1/4 >0. \qedhere\] \end{proof}

We are now in a position  to prove the main result of this section.

\begin{theorem}
 \label{thm:general}
Let $G$ be a locally compact, Hausdorff and \'etale groupoid.  Suppose that
$G$ is minimal and topologically principal.  Then   $\cs_r(G)$ is purely infinite
if and only if every nonzero positive element of $C_0(\go)$ is infinite in $\cs_r(G)$.
\end{theorem}

\begin{proof} The forward implication is trivial.  To see the reverse, let
 $a \in \cs_r(G)^+$ such that $a \neq 0$.  Using Lemma~\ref{lem:!} we can find a
{nonzero} $h \in C_0(\go)^+$ such that $h \precsim a$.  By assumption, we know $h$ is
infinite.  Since $\cs_r(G)$ is simple by \cite[Proposition II.4.6]{Renault},
$h$ is properly infinite by \cite[Proposition~3.14]{KR}.
Thus $a$ is properly
infinite by \cite[Lemma~3.8]{KR}, hence $\cs_r(G)$ is purely infinite.
\end{proof}

Recall that a Kirchberg algebra is a separable, nuclear, purely infinite simple $C^*$-algebra.
 We combine Theorem~\ref{thm:general} with
results from \cite{A-DR, BCFS, Renault}
 to obtain the following characterization of groupoid Kirchberg algebras.

\begin{cor}
 Let $G$ be a second-countable, locally compact, Hausdorff and \'etale groupoid. Then $\cs(G)$ is a
Kirchberg algebra if and only if $G$ is minimal, topologically principal, measure-wise amenable
and every non-zero positive element of $C_0(\go)$ is infinite in $\cs(G)$.
\end{cor}

\begin{proof}
Suppose $\cs(G)$ is a
Kirchberg algebra.  Then $\cs(G)$ is simple by definition and so $\cs(G)= \cs_r(G)$, $G$ is minimal and
$G$ topologically principal \cite[Theorem~5.1]{BCFS}.  Since $\cs(G)$ is nuclear, $\cs_r(G)$ is
also nuclear hence $G$ is measure-wise amenable by
\cite[Corollary~6.2.14]{A-DR}.
Finally, we apply Theorem~\ref{thm:general} to see that every non-zero positive element of
$C_0(\go)$ is infinite in $\cs(G)$.

Conversely, suppose $G$ is minimal, topologically principal, measure-wise amenable
and that every non-zero positive element of $C_0(\go)$ is infinite in $\cs(G)$.  Then
$\cs_r(G) = \cs(G)$ is nuclear by \cite[Corollary~6.2.14]{A-DR}, simple by \cite[Theorem~5.1]{BCFS}, separable
because $G$ is second countable \cite[Remark (iii) page 59]{Renault}
 and purely infinite
by Theorem~\ref{thm:general}.
\end{proof}

\section{$\cs$-algebras associated to ample groupoids}

In this section, we will restrict our attention to ample groupoids.
Although this might seem a very restrictive class of groupoids,
it actually includes a lot of important examples. Again, every Kirchberg algebra  in UCT
is Morita equivalent to a $\cs$-algebra associated to a Hausdorff, ample groupoid
(see \cite{spielberg2007}).
The ample case is far more manageable than the general case.
In particular there is a large number of projections in the associated algebra.
Let $G$ be a locally compact, Hausdorff and \'etale groupoid. If $G$ is ample, then the complex \emph{Steinberg algebra} associated to $G$ is
\[
A(G):=\lsp\{\chi_{_{B}}: B\text{~is a compact open bisection}\}\subseteq C_c(G)
\] where $\chi_{_B}$ denotes the characteristic function of $B$,
is dense in $C^*_r(G)$ see \cite[Proposition~4.2]{CFST} (see also \cite{steinberg}).
A quick computation shows that $\chi_{_{B}}*\chi_{_D}=\chi_{_{BD}}$ and $\chi_{_{B}}^*=\chi_{_{B\inv}}$,
so that if $B\subseteq \go$ is compact open, then $\chi_B$  is a projection.

\begin{theorem}
 \label{thm:totdis}
Let $G$ be a second countable, Hausdorff, ample groupoid.  Suppose that
$G$ is topologically principal, minimal and that $\Bb$ is a basis of $\go$ consisting
of compact open sets.
Then   $\cs_r(G)$ is purely infinite
if and only if every nonzero projection $p$ in $C_0(\go)$ with $\supp(p)  \in \Bb$ is infinite in $\cs_r(G)$.
\end{theorem}

\begin{proof}
The forward implication is trivial.
To see the reverse, suppose every nonzero projection $p$ of $C_0(\go)$ with $\supp(p) = U$ for some
$U \in \Bb$ is infinite in $\cs_r(G)$.  By Theorem~\ref{thm:general} it suffices to show that every positive element in
$C_0(\go)^+$ is infinite. Let $a \in \cs_r(\go)^+$ be {a nonzero element}.  We show that $a$ is
properly infinite.  We claim there is a nonzero projection $p \in {C_0(\go)^+}$ with $\supp(p) \subseteq U$ for some
$U \in \Bb$ such that
$p \precsim a$.  To see this, first note that characteristic functions of
the form $\chi_V$ are projections in $C_0(\go)$ for every compact open set $ V \subseteq\go$.
Since $\Bb$ is a basis of compact open sets, there exists a compact open set $U_0 \in \Bb$ and a nonzero $s \in \RR^+$ such that
$\chi_{_{U_0}}(x) \leq sa(x)$ for every $x \in \go$.
Then $p:=\chi_{_{U_0}} \leq sa$. Applying \cite[Proposition~2.7]{KR} we get that $p \precsim sa$ and so $p \precsim a$ as claimed.
 Since $p$ is infinite by assumption and $\cs_r(G)$ is simple,
$p$ is properly
infinite by \cite[Proposition~3.14]{KR}.  Hence $a$ is properly infinite by \cite[Lemma~3.8]{KR}.
\end{proof}

In the next corollary, we show how we can use the minimality of $G$ to strengthen our result.

\begin{cor}
\label{cor:tot_dis_nhd_base}
Let $G$ be a second countable,  Hausdorff, ample groupoid.  Suppose that
$G$ is topologically principal and minimal.
Then   $\cs_r(G)$ is purely infinite
if and only if there exists a point $x \in \go$ and a neighbourhood basis $\Dd$ at $x$ consisting of compact open sets so that
every nonzero projection $q$ in $C_0(\go)$ with $\supp(q)  \in \Dd$ is infinite in $\cs_r(G)$.
\end{cor}

\begin{proof}
Again, the forward direction is trivial.  For the reverse implication, suppose there exist a point $x \in \go$ and neighbourhood
basis $\Dd$ of $x$ {consisting of compact open sets} such that that every nonzero projection $q$  in $C_0(\go)$ with $\supp(q) \in \Dd$
is infinite in $\cs_r(G)$.  Let $\Bb$ be a basis of $\go$ consisting of compact open sets and suppose $p:=\chi_{_{U}}$
is a nonzero projection in $C_0(\go)$ with $ U  \in \Bb$.    By Theorem~\ref{thm:totdis}, it suffices to show that
$p$ is infinite.  Since
$G$ is minimal {and ample}, there exists a compact open bisection $B$ such that
$x \in s(B)$ and $r(B) \cap U \neq \emptyset$.  By shrinking $B$, we may assume that $r(B) \subseteq U$.
Since $s(B)$ is an {compact} open neighbourhood of $x$, there exists a $V \in \Dd$ such that $V \subseteq s(B)$.  By shrinking $B$
again, we may assume that $s(B) = V$. Thus,
\[\chi_{_V }= \chi_B^* \chi_{_{r(B)}} \chi_{_B}.\]
That is, $\chi_{_V}  \precsim \chi_{_r(B)}$.  Hence,  $\chi_{_r(B)}$ is properly infinite by \cite[Lemma~3.8]{KR}.
Finally, since $\chi_{_U} = \chi_{_r(B)} + \chi_{_{U - r(B)}}$, $\chi_{_U}$ is infinite.
\end{proof}

\section{An application:  $k$-graph $\cs$-algebras}
\label{Sec:kgraph}

In this section, we apply Theorem~\ref{thm:totdis} to  $\cs$-algebras associated to $k$-graphs. We assume the reader is familiar with the basic definitions and constructions of $k$-graphs and their $C^*$-algebras found in \cite{KP}, but we recall a few facts here. Let $\Lambda$ be a $k$-graph.  Then the associated $\cs$-algebra $\cs(\Lambda)$ is the universal $\cs$-algebra generated by a Cuntz-Krieger $\Lambda$-family $\{s_{\lambda}^{}: \lambda \in \Lambda\}$. To keep things clean, we will restrict our attention to row-finite $k$-graphs with no sources but similar
results hold in the more general setting. We think our results will be useful in this setting because
necessary and sufficient conditions on $\Lambda$ for $\cs(\Lambda)$ to be purely infinite simple are not known.

Following \cite{KP} we recall how $\cs(\Lambda)$ can be realised  as the $\cs$-algebra of a second countable, Hausdorff, ample  groupoid
$G_{\Lambda}$ as follows.  Let $\Lambda^{\infty}$ denote the infinite path space of $\Lambda$ and $\Lambda^\infty(v)$ be the set of infinite paths with range $v$. Define
\[
 G_{\Lambda}:= \{(x,n,y) \in \Lambda^{\infty} \times \NN^k \times \Lambda^{\infty} : \sigma^l(x)=\sigma^m(y), n=l-m\}
\]
where $\sigma$ is the shift map.  We view $(x,n,y)$ as a morphism with source $y$ and range $x$.
Composition is given by $(x,n,y)(y,m,w)= (x,n+m,w)$.
The unit space $G_{\Lambda}^{(0)}$ is identified with $\Lambda^{\infty}$.
For $\lambda, \mu \in \Lambda$ with $s(\lambda)=s(\mu)$ we define
\[
 Z(\lambda, \mu):= \{(\lambda z, d(\lambda)-d(\mu),\mu z) : z \in \Lambda^{\infty}(s(\lambda))\}.
\]
The (countable) collection of all such  $Z(\lambda, \mu)$
generate a topology under which $G_{\Lambda}$ is a second countable, Hausdorff, ample groupoid by \cite[Proposition~2.8]{KP}.
Further, the relative topology on the unit space $\Lambda^{\infty}$ has a basis of compact cylinder sets
\[
 Z(\lambda) := \{\lambda x \in \Lambda^{\infty} : x \in \Lambda^{\infty}(s(\lambda))\}
\] by identifying $Z(\lambda,\lambda)$ and $Z(\lambda)$ from \cite[Lemma~2.6 and Proposition~2.8]{KP}.
Note that $G_{\Lambda}$ is amenable by \cite[Theorem~5.5]{KP} and hence $\cs_r(G_{\Lambda})= \cs(G_{\Lambda})$.
{It was shown in \cite{KP} that $$\cs(\Lambda) \cong \cs(G_{\Lambda}).$$
More specifically, by \cite[Corollary~3.5(i)]{KP}, there is a (unique)} isomorphism $\phi: \cs(\Lambda) \to \cs(G_{\Lambda})$
such that $\phi(s_{\lambda}^{}) = \chi_{_{Z(\lambda, s(\lambda))}}$.
Note that
\begin{equation*}
\phi(s_{\mu}^{}s_{\mu}^*)= \chi_{_{Z(\mu, s(\mu))}}\chi^*_{_{Z(\mu, s(\mu))}}= \chi_{_{Z(\mu, s(\mu))}}\chi_{_{Z(s(\mu), \mu)}}
= \chi_{_{Z(\mu, \mu)}}=\chi_{_{Z(\mu)}}.
\end{equation*}
With all of this theory in place, along with the simplicity results of \cite{RS07} and \cite{BCFS},
the following is an immediate corollary of Theorem~\ref{thm:totdis} and Corollary~\ref{cor:tot_dis_nhd_base}.

\begin{cor}
\label{cor:kgraph}
Suppose $\Lambda$ is a row-finite $k$-graph with no sources such that $\Lambda$ is
aperiodic and cofinal {in the sense of \cite{RS07}}.  Then
\begin{enumerate}
\item \label{it1:graph} For $\mu\in \Lambda$, $s_{\mu}^{}s_{\mu}^*$ is infinite if and only if $s_{s(\mu)}$ is.
\item \label{it2:graph} $\cs(\Lambda)$ is purely infinite simple if and only
if {$s_v$ is infinite for every $v \in \Lambda^0$}.
\item \label{it3:graph} $\cs(\Lambda)$ is purely infinite simple if and only
if there exists $x \in \Lambda^{\infty}$ such that $s_v$ is infinite for every vertex $v$ on $x$.
\end{enumerate}
\end{cor}

\begin{proof}
For (\ref{it1:graph}), we use a trick used in \cite[Lemma~8.13]{Sims}. Recall that
infiniteness is preserved under von Neumann equivalence, hence $s_{\mu}^{}s_{\mu}^*$ is infinite if and
only if $s_{\mu}^*s_{\mu}^{}=s_{s(\mu)}^{}$ is infinite.

For (\ref{it2:graph}),  we apply Theorem~\ref{thm:totdis}  to the second countable, Hausdorff, ample groupoid  $G_{\Lambda}$; first we check the
remaining hypotheses of Theorem~\ref{thm:totdis}. Since $\Lambda$ is cofinal and aperiodic, $\cs(G_\Lambda)\cong\cs(\Lambda)$  is simple by \cite[Theorem~3.1]{RS07}.
Thus $\cs(G_{\Lambda})= \cs_r(G_{\Lambda})$ is simple and hence $G_{\Lambda}$ is topologically
principal and minimal by \cite[Theorem~5.1]{BCFS}.    

We have that the collection of cylinder
sets of the form $Z(\mu)$ form a basis $\Bb$ of $G_{\Lambda}^{(0)}$ consisting of compact open sets.
Now we apply Theorem~\ref{thm:totdis} to see that $\cs(G_{\Lambda})$ is purely infinite  if and only
if each $\chi_{_{Z(\mu)}}$ is infinite. Let $\phi: C^*(\Lambda)\to C^*(G_\Lambda)$ be the isomorphism characterized by $s_\mu\mapsto \chi_{_{Z(\mu)}}$. Since $\phi$ is an isomorphism, this gives
$\chi_{_{Z(\mu)}}$ is infinite if and only if
$\phi\inv(\chi_{_{Z(\mu)}})=s_{\mu}^{}s_{\mu}^*$ is infinite.  Finally,   $s_{\mu}^{}s_{\mu}^*$ if and only if $s_{s(\mu)}$ is infinite by (1).

For (\ref{it3:graph}), given an infinite path $x$, the collection of compact open sets of the form $Z(x(0,n))$ for $n \in \NN^k$
form a neighbourhood base at $x$.  Now proceed as in the proof of (2) replacing Theorem~\ref{thm:totdis} with Corollary~\ref{cor:tot_dis_nhd_base} and $\mu$ with $x(0,n)$.
\end{proof}


\section{The non-simple case}
\label{sec:NS}

Let $A$ be a $C^*$-algebra.  A pair of positive elements $(a_1,a_2)\in A\times A$ has the
{\em matrix diagonalization property in $A$} in the sense of \cite[Definition 3.3.]{KS} if
for every positive matrix $\minimatrix{a_1}{b_{12}}{b_{21}}{a_2}$ with $b_{ij}\in A$ and every $\epsilon_1,
\epsilon_2, \delta>0$ there exists $d_1, d_2\in A$ with
\[
\|d_i^* a_i d_i-a_i\|<\epsilon_i \text{ and }  \quad \|d_i^*b_{ij}d_j\|<\delta.
\]
A subset $\Ff$ of $A^+$ is a {\em filling family} for $A$, in the sense of \cite[Definition 3.10]{KS},
 if for every hereditary $\cs$-subalgebra $H$ of $A$ and every primitive ideal $I$ of $A$ with $H \not\subseteq I$
 there exist $f \in \Ff$ and $z \in A$ with $z^*z \in H$ and $zz^* = f \not\in I$.

By Proposition~3.13 and Lemma~3.12 of \cite{KS}, if $A^+$ contains a filling family $\Ff$ that
is closed under  $\epsilon$-cut-downs and every pair of elements $(a_1,a_2)\in \Ff\times \Ff$ has the matrix
diagonalization property,
then $A$ is strongly purely infinite.  In this section we provide a characterization of when the reduced groupoid $C^*$-algebra
is strongly purely infinite (Proposition~\ref{cor.1.2}).
In our proof of Proposition~\ref{cor.1.2} we will use results from \cite{BCFS} to
describe ideals of reduced groupoid $C^*$-algebras.   First we need the following lemma.
Recall that a subset $D \subseteq \go$ is said to be {\em invariant} if
$G\cdot D:=\{r(\gamma): s(\gamma)\in D\}\subseteq D$.

\begin{lem}
\label{lem.1.1}
Let $G$ be a second countable, locally compact, Hausdorff and \'etale groupoid such that
 $C^*(G) = C^*_r(G)$. Then the following properties are equivalent:
\begin{enumerate}
\item \label{it1:lem1.1} For every closed invariant set $D \subseteq \go$
$$C^*(G|_{D})= C^*_r(G|_{D}).$$
\item \label{it2:lem1.1} For every closed invariant set $D \subseteq \go$ the sequence
$$\xymatrix{0 \ar[r] & C^*_r(G|_{\go-D}) \ar[r]^-{\iota_r} & C^*_r(G) \ar[r]^-{\rho_r}& C^*_r(G|_{D}) \ar[r]& 0}$$
is exact where $\iota_r$ and $\rho_r$ are determined on continuous functions by extension and restriction
respectively.
\end{enumerate}
\end{lem}

\begin{rmk}
             In \cite[Remark~4.10]{Renault91}, Renault mentions that if $G|_{D}$ is amenable for every closed invariant set $D \subseteq \go$, then item (\ref{it2:lem1.1}) of Lemma~\ref{lem.1.1} follows.  Thus Lemma~\ref{lem.1.1} is a strengthening of Renault's comment.
            \end{rmk}

\begin{proof}
Fix a closed invariant set $D \subseteq \go$ and let $U=\go-D$.  Consider the following diagram:
\begin{equation}
\label{diagram2.0}
\xymatrix{0 \ar[r] & C^*(G|_{U}) \ar[d]_{\pi_{U}} \ar[r]^-{\iota} & C^*(G) \ar[d]_\pi \ar[r]^-{\rho}& C^*(G|_{D}) \ar[d]_{\pi_{D}} \ar[r]& 0\\
0 \ar[r] & C^*_r(G|_{U}) \ar[r]^-{\iota_r} & C^*_r(G) \ar[r]^-{\rho_r}& C^*_r(G|_{D}) \ar[r]& 0}
\end{equation}
 where $\pi_U, \pi$ and $\pi_D$ are the respective quotient maps, and $\iota, \iota_r$ and $\rho, \rho_r$ extend extension
and restriction respectively.  Since all of the maps involved are continuous, the diagram commutes.
We also have that the top row of (\ref{diagram2.0}) is exact by \cite[Lemma~2.10]{MRW96}.

(\ref{it2:lem1.1})$\Rightarrow$(\ref{it1:lem1.1}):
We show the surjective map $\pi_{D}$ is injective. Fix any $a\in C^*_r(G|_{D})$ with $\pi_D(a)=0$. Find $b\in C^*(G)$ with $\rho(b)=a$. From $\pi_D(\rho(b))=\rho_r(\pi(b))=0$,
 exactness of (\ref{diagram2.0}), surjectivity of $\pi_{U}$, and $\iota\circ\pi_{U}=\pi\circ\iota$ we obtain
\[
\pi(b)\in\ker \rho_r=\iota_r(C^*_r(G|_{U}))=\iota_r\circ\pi_{U}(C^*(G|_{U}))=\pi\circ\iota(C^*(G|_{U})).
\]
Find $c\in C^*(G|_{U})$ with $\pi(b)=\pi\circ\iota(c)$. As $\pi$ is an isomorphism by assumption we obtain that $b=\iota(c)$.
 Hence $a=\rho(b)=\rho\circ\iota(c)=0$, and $C^*(G|_{D}) = C^*_r(G|_{D})$.

(\ref{it1:lem1.1})$\Rightarrow$(\ref{it2:lem1.1}): By assumption the maps $\pi$ and $\pi_D$ are
isomorphisms.  Using the commutative diagram (\ref{diagram2.0}) and the exactness of the top line of that diagram, the exactness
of the bottom line follows from a easy diagram chase.
\end{proof}

Let $G$ be a second countable, locally compact, Hausdorff and \'etale groupoid and
$D$ be a closed invariant set of $\go$ and define $U=\go-D$.  Recall from Lemma~\ref{lem:new}(\ref{it4:review})
we have the communing diagram:
\begin{equation}
\label{com diag}
\xymatrix{0 \ar[r] & C^*_r(G|_{U}) \ar[d]_{E_{U}} \ar[r]^-{\iota_r} & C^*_r(G) \ar[d]_E \ar[r]^-{\rho_r}& C^*_r(G|_{D}) \ar[d]_{E_{D}} \ar[r]& 0\\
0 \ar[r] & C_0({U}) \ar[r]^-{\iota_0} & C_0(\go) \ar[r]^-{\rho_0}& C_0({D}) \ar[r]& 0.}
\end{equation}
Notice that the bottom row in \eqref{com diag} is exact.  We will use  this diagram several times.
We also use the notation  $\I{S}$ for the ideal in $C^*_r(G)$ generated by $S\subseteq C^*_r(G)$.

\begin{prop}
\label{cor.1.2}
Let $G$ be a second countable, locally compact, Hausdorff and \'etale groupoid such that $C^*(G) = C^*_r(G)$. Then the following properties are equivalent:
\begin{enumerate}
\item \label{it1:cor1.2} The $\cs$-algebra $C^*_r(G)$ is strongly purely infinite, and for every ideal $I$ in $C^*_r(G)$,
\[
I=\I{I\cap C_0(\go)}.
\]
\item \label{it2:cor1.2} For every closed invariant set $D \subseteq \go$, $G|_{D}$ is topologically principal; the sequence

\begin{equation}
 \label{exact}
\xymatrix{0 \ar[r] & C^*_r(G|_{U}) \ar[r]^-{\iota_r} & C^*_r(G) \ar[r]^-{\rho_r}& C^*_r(G|_{D}) \ar[r]& 0}
\end{equation}
is exact where $U = \go - D$, $\iota_r$ and $\rho_r$ are determined on continuous functions by extension and restriction
respectively;
and for every pair of elements $a,b$ in $C_0(\go)^+$ the 2-tuple $(a,b)$ has the matrix diagonalization property in $C^*_r(G)$.
\end{enumerate}
\end{prop}

\begin{proof}
(\ref{it1:cor1.2})$\Rightarrow$(\ref{it2:cor1.2}): Fix a closed invariant set
$D \subseteq \go$ and $U=\go-D$. For this $D$ and $U$ we have a commuting diagram as
in \eqref{com diag}.
Define $I:=\ker\rho_r\subseteq C^*_r(G)$. Using the diagram, $\rho_0(E(I))=E_{D}(\rho_r(I))=0$, implying that
$E(I)\subseteq \iota_0(C_0({U}))$. Since $E(b)=b$ for $b\in C_0(\go)$, $I\cap C_0(\go)\subseteq E(I)$.
Using assumption (\ref{it1:cor1.2}) we have $I=\I{I\cap C_0(\go)}$. Hence
\[
\ker\rho_r=I=\I{I\cap C_0(\go)}\subseteq \I{E(I)}\subseteq \I{\iota_0(C_0(U))}\subseteq
\iota_r(C^*_r(G|_{U}));
\] that is $\ker\rho_r \subseteq \operatorname{image}(\iota_r)$.
Thus \eqref{exact} is exact.

We know that each $G|_{D}$ is topologically principal by \cite[Remark 5.10]{BCFS} provided that
$C^*(G|_{D}) = C^*_r(G|_{D})$. The latter follows from Lemma \ref{lem.1.1} since \eqref{exact} is exact.

Since $C^*_r(G)$ is strongly purely infinite, Lemma~5.8 in \cite{KR-infty} implies that every
pair $(a,b)$ of positive elements in $C_0(\go)$  has the matrix diagonalization property in $C^*_r(G)$.

(\ref{it2:cor1.2})$\Rightarrow$(\ref{it1:cor1.2}): Since we assumed that $G|_D$ is topologically principal for all closed
 invariant  $D\subseteq \go$,  by the proof of Corollary~5.9 in \cite{BCFS} (see also
\cite[Remark~5.10]{BCFS}), we know $I=\I{I\cap C_0(\go)}$ for every ideal $I$ in $C^*_r(G)=C^*(G)$ provided
that  $C^*(G|_{D}) = C^*_r(G|_{D})$ for every closed invariant set $D \subseteq \go$. But this follows from Lemma~\ref{lem.1.1} since $C^*(G) = C^*_r(G)$ and \eqref{exact} is
exact, which are assumed in (\ref{it2:cor1.2}). Hence (\ref{it2:cor1.2}) implies $I=\I{I\cap C_0(\go)}$.

We prove $C^*_r(G)$ is strongly purely infinite. Define $\Ff:=C_0(\go)^+\subseteq C^*_r(G)$.
By functional calculus we know
\[
f(a)\in \Ff,\ \ \ \text{for}\ \ \ f\in C_0(\RR)^+, \ \ \ a\in \Ff.
\]
In particular $\Ff$ is closed under $\varepsilon$-cut-downs, i.e., for each $a\in \Ff$,
and $\varepsilon\in(0,\|a\|)$ we have $(a-\varepsilon)_+\in \Ff$. By (\ref{it2:cor1.2}) each pair $(a,b)$
 with $a,b\in\Ff$ has the matrix diagonalization property (of \cite[Definition 3.3]{KS}).
Now by Lemma 3.12 of \cite{KS} we know that $\Ff$ has the matrix diagonalization property
of \cite[Definition 3.10(ii)]{KS}. If follows from Proposition 3.13 of \cite{KS} that $C^*_r(G)$
is strongly purely infinite provided that $\Ff$ is a filling family for $C^*_r(G)$, which we now show.

Fix any hereditary $\cs$-subalgebra $H$ of $C^*_r(G)$ and any ideal $I$ of $C^*_r(G)$ with
$H \not\subseteq I$. We know $I=\I{I\cap C_0(\go)}$, hence $I=\iota_r(C^*_r(G|_{U}))$ for some open invariant set $U\subseteq \go$.
Let $D=\go-U$
 and consider the corresponding commuting diagram \eqref{com diag}.

Select $d\in H^+$, $d\not\in I$.  Define $c:=\rho_r(d)$. As $d\notin I=\ker\rho_r$ by  exactness in (2),
we know $\rho(d)\neq 0$. Since $E_D$ is faithful and $c$ positive,
\[
\epsilon := \frac{1}{4}\|E_D(c)\|>0.
\]
By (\ref{it2:cor1.2}) the groupoid $G|_{D}$ is
topologically principal, hence Lemma~\ref{lem:star} gives $f \in C_0(D)^+$ such that
\[
\|f\| = 1, \ \ \|fcf - fE_D(c)f\| < \epsilon, \ \ \|fE_D(c)f\| > \|E_D(c)\| - \epsilon.
\]
Recall \cite[Lemma~2.2]{KR-infty}: For $x,y$ positive and $\delta>0$ with $\|x-y\|<\delta$ there
exist a contraction $a$ with $a^*xa=(y-\delta)_+$. Use this to find a contraction $a \in \cs_r(G|_{D})$ such that
$$h:= a^*fcfa= (fE_D(c)f - \epsilon)_+ \in C_0(D)^+.$$
Notice that \[\|h\| \geq \|fE_D(c)f\| - \epsilon > \|E_D(c)\| - 2\epsilon>0.\] Using that $\rho_r$ restricts to
the map $C_0(G^{(0)})\to C_0(D)$, select $b\in C_0(\go)^+$ such that $\rho_r(b)=h$. Also as $\rho_r$ is surjective
find $w\in C^*_r(G)$ such that $\rho_r(w)=fa$. Since $\rho_r(b-w^*dw)=h-a^*fcfa=0$ we have $b=w^*dw+v$ for some $v\in I$.
Let $\{e_\lambda\}$ denote an approximate unit of $I=C^*_r(G|_{U})$ with $e_\lambda\in C_0({U})$ (see Lemma~\ref{lem:new}). Let $1$ be the unit of the
unitisation of  $C^*_r(G)$. Then $(1-e_\lambda)v(1-e_\lambda)\to 0$. For suitable $\lambda_0$
and $e:=1-e_{\lambda_0}$ we get \[\|ew^*dwe-ebe\|=\|eve\|< \epsilon.\] Use Lemma~2.2 of \cite{KR-infty} to find a contraction $u\in C^*_r(G)$ such that
\[
g:= u^*ew^*dweu= (ebe - \epsilon)_+ \in C_0(\go)^+=\Ff.
\]
Since $be_{\lambda_0}+e_{\lambda_0}b-e_{\lambda_0}be_{\lambda_0}\in C_0({U})\subseteq \ker\rho_r$
we obtain that $\rho_r(ebe)=\rho_r(b)=h$. Moreover by functional calculus we know $(h - \epsilon)_+=(fE_D(c)f - 2\epsilon)_+$. We conclude
\[
\|\rho_r(g)\|=\|(h - \epsilon)_+\|=\|(fE_D(c)f - 2\epsilon)_+\|\geq \|fE_D(c)f\| - 2\epsilon>\|E_D(c)\| - 3\epsilon>0,
\]
ensuring $g\not\in I$. Finally with $z:=u^*ew^*d^{1/2}\in C^*_r(G)$ we obtain $g=z{z}^*$ and
${z}^*z\in H$. By definition $\Ff$ is a filling family for $C^*_r(G)$ completing the proof.
\end{proof}



\end{document}